\documentclass{amsart}
\usepackage{amsmath}
\usepackage{amssymb}
\usepackage{amsthm}

\DeclareMathOperator{\dvol}{dvol}

\DeclareMathOperator{\Ric}{Ric}

\DeclareMathOperator{\sn}{sn}
\DeclareMathOperator{\cn}{cn}

\newcommand{\og}{\overline{g}}

\newcommand{\ov}{\overline{v}}

\newcommand{\onabla}{\overline{\nabla}}

\newcommand{\lp}{\langle}
\newcommand{\rp}{\rangle}
\newcommand{\lv}{\lvert}
\newcommand{\rv}{\rvert}
\newcommand{\lV}{\lVert}
\newcommand{\rV}{\rVert}

\newcommand{\bN}{\mathbb{N}}

\newcommand{\bR}{\mathbb{R}}

\newcommand{\comment}[1]{}

\newtheorem{thm}{Theorem}[section]
\newtheorem{prop}[thm]{Proposition}
\newtheorem{lem}[thm]{Lemma}

\theoremstyle{definition}
\newtheorem{defn}[thm]{Definition}

\newtheorem{example}[thm]{Example}

\theoremstyle{remark}

\numberwithin{equation}{section}
\usepackage{esint}

\begin{document}

\title[Cai--Galloway splitting theorem for SMMS]{A generalization of the Cai--Galloway splitting theorem to smooth metric measure spaces}
\author{Jeffrey S. Case}
\thanks{JSC was partially supported by NSF-DMS Grant No.\ 1004394}
\address{Department of Mathematics \\ 1004 Fine Hall\\ Princeton University \\ Princeton, NJ 08544}
\email{jscase@math.princeton.edu}
\author{Peng Wu}
\address{Department of Mathematics \\ 583 Malott Hall \\ Cornell University \\ Ithaca, NY 14853}
\email{wupenguin@math.cornell.edu}
% \date{}
\begin{abstract}
We generalize the splitting theorem of Cai--Galloway for complete Riemannian manifolds with $\Ric\geq-(n-1)$ admitting a family of compact hypersurfaces tending to infinity with mean curvatures tending to $n-1$ sufficiently fast to the setting of smooth metric measure spaces.  This result complements and provides a new perspective on the splitting theorems recently proven by Munteanu--Wang and Su--Zhang.  We show that the mean curvature assumption in our result is sharp, which also provides an example showing that the assumption $R\geq-(n-1)$ in the Munteanu--Wang splitting theorem for expanding gradient Ricci solitons cannot be relaxed to $R>-n$.  We also use our result to study a certain class of conformally compact quasi-Einstein metrics, giving, as generalizations of respective results of Cai--Galloway and Lee, necessary conditions for the boundary to be connected and for the bottom of the spectrum of the weighted Laplacian to be maximal.
\end{abstract}
\maketitle

%%%%%%%%%%%%%%%%%%%%%%%%%%%%%%%%%%%%%%%%%%%%%%%%%%%%%%%%%%%%%%%%%
%                                                               %
% Structure of the document                                     %
%                                                               %
% 1. Intro                                                      %
% *. Acknowledgments                                            %
% 2. Smooth metric measure spaces                               %
% 3. Sharpness                                                  %
% 4. The splitting theorem                                      %
%                                                               %
%%%%%%%%%%%%%%%%%%%%%%%%%%%%%%%%%%%%%%%%%%%%%%%%%%%%%%%%%%%%%%%%%

\section{Introduction}
\label{sec:intro}

Smooth metric measure spaces are four-tuples $(M^n,g,e^{-\phi}\dvol,m)$ of a Riemannian manifold together with a choice of smooth measure $e^{-\phi}\dvol$ --- that is, $\phi\in C^\infty(M)$ and $\dvol$ is the usual Riemannian volume element determined by $g$ --- and dimensional parameter $m\in[0,\infty]$.  These spaces admit a natural generalization of the Ricci curvature known as the ($m$-)Bakry-\'Emery Ricci tensor, namely the tensor
\[ \Ric_\phi^m := \Ric + \nabla^2\phi - \frac{1}{m}d\phi\otimes d\phi . \]
Smooth metric measure spaces and the Bakry-\'Emery Ricci curvature play prominent roles in the study of the Ricci flow and in optimal transport, and for this reason have begun to attract lots of attention.  In particular, there has been much effort expended on generalizing results from comparison geometry to smooth metric measure spaces with Bakry-\'Emery Ricci curvature bounded below.

An important result in the AdS/CFT correspondence is that Poincar\'e-Einstein manifolds with conformal infinity of nonnegative Yamabe type are connected at infinity~\cite{CaiGalloway1999,WittenYau1999}.  Underlying this fact is the following splitting theorem obtained by Cai and Galloway~\cite{CaiGalloway1999}.

\begin{thm}[Cai--Galloway~\cite{CaiGalloway1999}]
\label{thm:cg}
(1) Let $(M^n,g)$ be a complete Riemannian manifold with compact boundary $\Sigma=\partial M$ such that $\Ric\geq-(n-1)g$ and $H>(n-1)$ for $H$ the mean curvature of $\Sigma$ with respect to the outward pointing normal.  Then $M$ is compact.

(2) Let $(M^n,g)$ be a complete Riemannian manifold with $\Ric\geq-(n-1)g$ and suppose there exists a sequence $\{\Sigma_k\}$ of compact hypersurfaces such that, for any fixed base point $o\in M$,
\begin{enumerate}
\item each $\Sigma_k$ separates $M$ into two disjoint sets,
\item $d(o,\Sigma_k)\to\infty$ as $k\to\infty$, and
\item the constants
\[ h_k := \inf_{x\in\Sigma_k} \left\{ H_{\Sigma_k}(x), n \right\} \]
satisfy
\begin{equation}
\label{eqn:cg_mean_curvature}
\lim_{k\to\infty} \left(n-h_k\right)e^{2d(o,\Sigma_k)} = 0 .
\end{equation}
\end{enumerate}
Then either $M$ has one end or $M$ is isometric to
\[ \left( \bR\times\Sigma, dt^2\oplus e^{2t}h \right) \]
for some compact Riemannian manifold $(\Sigma^{n-1},h)$ with nonnegative Ricci curvature.
\end{thm}

This result is sharp, in the sense that the assumption~\eqref{eqn:cg_mean_curvature} cannot be removed.  Indeed, if one only assumes that the limit is finite, then there are additional complete Riemannian manifolds with multiple ends which meet the hypotheses of the theorem.  For details, see~\cite{CaiGalloway1999} or Section~\ref{sec:sharp}.

Shortly after Theorem~\ref{thm:cg} was proven, X.\ Wang~\cite{Wang2001} obtained a similar rigidity result for conformally compact manifolds by using the bottom of the spectrum of the Laplacian in place of the mean curvature of hypersurfaces, and then Li and J.\ Wang~\cite{LiWang2001} removed the conformal compactness assumption.  This work then led Li and Wang~\cite{LiWang2002} to reformulate Cheng's estimate~\cite{Cheng1975} as a rigidity result.

\begin{thm}[Li--Wang~\cite{LiWang2002}]
\label{thm:lw}
Let $(M^n,g)$ be a complete, noncompact Riemannian manifold with $\Ric\geq-(n-1)g$.  Then $\lambda_1(-\Delta)\leq\frac{(n-1)^2}{4}$.  Moreover, if equality holds, then either
\begin{enumerate}
\item $M$ has one end,
\item $(M^n,g)$ is isometric to $(\bR\times N^{n-1},dt^2\oplus e^{2t}h)$ for some compact Riemannian manifold $(N^{n-1},h)$ with nonnegative Ricci curvature, or
\item $n=3$ and $(M^3,g)$ is isometric to $(\bR\times N^2,dt^2\oplus\cosh^2t\,h)$ for some compact Riemannian manifold $(N^2,h)$ with scalar curvature $R_h\geq-2$.
\end{enumerate}
\end{thm}

The key step in the proof of Theorem~\ref{thm:cg} is to show that the Busemann function $\beta$ associated to the sequence of hypersurfaces $\{\Sigma_k\}$ satisfies $\Delta_\phi\beta\geq1$ in a suitable sense.  This function can also be used to show that a Riemannian manifold as in Theorem~\ref{thm:cg} achieves equality in Cheng's estimate for the first eigenvalue of the Laplacian, and thus Theorem~\ref{thm:cg} can be seen as a special case of Theorem~\ref{thm:lw}.

Recently Munteanu and J.\ Wang~\cite{MunteanuWang2011,MunteanuWang2011b} generalized many aspects of the work of Li and Wang~\cite{LiWang2001,LiWang2002} to the setting of smooth metric measure spaces with $\infty$-Bakry-\'Emery Ricci curvature bounded below, including two generalizations of Theorem~\ref{thm:lw}.  At the same time, Su and Zhang~\cite{SuZhang2011} proved the analogue of Theorem~\ref{thm:lw} for smooth metric measure spaces with $m$-Bakry-\'Emery Ricci tensor bounded below for $m<\infty$.  More precisely, their respective works combine to establish the following result.

\begin{thm}[Munteanu--Wang~\cite{MunteanuWang2011}, Su--Zhang~\cite{SuZhang2011}]
\label{thm:weighted_cheng}
Let $(M^n,g,e^{-\phi}\dvol,m)$ be a complete, noncompact smooth metric measure space with $\Ric_\phi^m\geq-\frac{1}{m+n-1}g$.  If $m=\infty$, assume additionally that $\lv\nabla\phi\rv^2\leq 1$.  Then $\lambda_1(-\Delta_\phi)\leq\frac{1}{4}$.  Moreover, if equality holds, one of the following is true:
\begin{enumerate}
\item $M$ has one end.
\item $(M^n,g,e^{-\phi}\dvol,m)$ is isometric to
\[ \left( \bR\times N^{n-1}, dt^2\oplus e^{\frac{2t}{m+n-1}}h, e^{\frac{mt}{m+n-1}}\dvol, m \right) \]
for some compact Riemannian manifold $(N^{n-1},h)$ with nonnegative Ricci curvature.
\item $n=3$, $m=0$, and $(M^n,g,e^{-\phi}\dvol,m)$ is isometric to
\[ (\bR\times N^2,dt^2\oplus(2\cosh\frac{t}{2})^2\,h, \dvol, 0) \]
for some compact Riemannian manifold $(N^2,h)$ with scalar curvature $R_h\geq-2$.
\end{enumerate}
\end{thm}

In the limiting case $m=\infty$, Munteanu and Wang~\cite{MunteanuWang2011} used this result to show that the only steady gradient Ricci solitons with more than one end are the cylinders
\[ \left( \bR\times N^{n-1}, dt^2\oplus h, e^t\dvol, \infty\right) \]
constructed by taking products with a compact Ricci flat manifold $(N^{n-1},h)$.  By refining their techniques, they were able to give in a later article~\cite{MunteanuWang2011b} a similar characterization of expanding gradient Ricci solitons with more than one end.  More precisely, we say that $(M^n,g)$ is an expanding gradient Ricci soliton if there exists a function $\phi\in C^\infty(M)$, called the soliton potential, such that $\Ric_\phi^\infty=-g$.  Munteanu and Wang~\cite{MunteanuWang2011b} then established the following result.

\begin{thm}
\label{thm:munteanu_wang2011b}
Let $(M^n,g)$ be an expanding gradient Ricci soliton with soliton potential $\phi$.  If the scalar curvature $R\geq-(n-1)$ and $M$ has more than one end, then there is a compact Einstein manifold $(N^{n-1},h)$ with $\Ric_h=-h$ such that $(M^n,g,e^{-\phi}\dvol)$ is isometric to
\[ \left( \bR\times N^{n-1}, dt^2\oplus h, e^{\frac{t^2}{2}}\dvol \right) . \]
\end{thm}

It is known that the scalar curvature of any expanding gradient Ricci soliton satisfies $R\geq-n$, and the inequality is strict unless $\phi$ is constant~\cite{PigolaRimoldiSetti2011}.  Some assumption on the scalar curvature is necessary, as there are many Einstein metrics with negative scalar curvature and multiple ends, but it was left open in~\cite{MunteanuWang2011b} to what extent the assumption $R\geq-(n-1)$ in Theorem~\ref{thm:munteanu_wang2011b} is necessary.

In light of Theorem~\ref{thm:weighted_cheng}, it is natural to ask whether Theorem~\ref{thm:cg} can be generalized to the setting of smooth metric measure spaces, and if so, whether the assumption~\eqref{eqn:cg_mean_curvature} might shed some light on the necessity of the assumption $R\geq-(n-1)$ in Theorem~\ref{thm:munteanu_wang2011b}.  The first goal of this article is to show that this is indeed the case.  More precisely, we will prove the following generalization of Theorem~\ref{thm:cg}.

\begin{thm}
\label{thm:splitting}
Let $(M^n,g,e^{-\phi}\dvol,m)$ be a smooth metric measure space with $\Ric_\phi^m\geq -\frac{1}{m+n-1}g$; if $m=\infty$, assume additionally that $\lv\nabla\phi\rv^2\leq1$.  Suppose also that there exists a sequence $\{\Sigma_k\}$ of compact hypersurfaces such that, for any fixed base point $o\in M$,
\begin{enumerate}
\item each $\Sigma_k$ separates $M$ into two disjoint sets,
\item $d(o,\Sigma_k)\to\infty$ as $k\to\infty$, and
\item the quantity
\[ h_k = \min_{x\in\Sigma_k}\left\{ H_\phi(x), 1\right\}, \]
where $H_\phi$ is the weighted mean curvature of $\Sigma_k$, satisfies
\begin{equation}
\label{eqn:splitting_mean_curvature}
\lim_{k\to\infty} (1-h_k)e^{\frac{2d(o,\Sigma_k)}{m+n-1}} = 0 .
\end{equation}
\end{enumerate}
Then either $M$ has one end, or $M$ is isometric to the product
\begin{equation}
\label{eqn:splitting_conclusion}
\left( \bR\times N, dt^2\oplus e^{\frac{2t}{m+n-1}}h, e^{\frac{mt}{m+n-1}}\dvol, m \right)
\end{equation}
for some compact Riemannian manifold $(N^{n-1},h)$ with nonnegative Ricci curvature.
\end{thm}

Moreover, we will show in Section~\ref{sec:sharp} that Theorem~\ref{thm:splitting} is sharp in the same sense as Theorem~\ref{thm:cg}: If one relaxes the assumption~\eqref{eqn:splitting_mean_curvature} to only require that the limit is finite, then there are additional examples with two ends which do not split isometrically as~\eqref{eqn:splitting_conclusion}.  In particular, Example~\ref{ex:besse} gives such an example which, in the limiting case $m=\infty$, is a complete expanding gradient Ricci soliton with $\Ric_\phi^\infty=-g$ such that $R\to-n$ at infinity in one of its ends.  This shows that it is not enough to simply assume that $R>-n$ in Theorem~\ref{thm:munteanu_wang2011b}; equivalently, it is not enough to consider gradient Ricci solitons with nonconstant soliton potentials.

There are three key aspects of Theorem~\ref{thm:splitting} which we wish to emphasize.  First, Theorem~\ref{thm:splitting} can be regarded as a rigidity result; in Section~\ref{sec:splitting} we will show that if any hypersurface in Theorem~\ref{thm:splitting} satisfies $H_\phi>1$, then $M$ has one end.  Second, Theorem~\ref{thm:splitting} makes sense in the limit $m=\infty$, where it is a statement about complete smooth metric measure spaces with nonnegative $\infty$-Bakry-\'Emery Ricci tensor which admit a sequence of hypersurfaces tending to infinity having weighted mean curvature tending to one.  Third, we will prove Theorem~\ref{thm:splitting} by modifying the proof of Theorem~\ref{thm:cg} given by Cai and Galloway~\cite{CaiGalloway1999}, which amounts to understanding the behavior of the Busemann function associated to the sequence $\{\Sigma_k\}$.  In particular, our proof does not use Theorem~\ref{thm:weighted_cheng}, though an idea from our proof can also be used to show that Theorem~\ref{thm:splitting} is nevertheless a consequence of Theorem~\ref{thm:weighted_cheng}.

The second goal of this article is to show that Theorem~\ref{thm:splitting} can be used to prove a connectedness result for conformally compact quasi-Einstein smooth metric measure spaces analogous to the one established in~\cite{CaiGalloway1999,WittenYau1999}.  These are smooth metric measure spaces $(M^n,g,e^{-\phi}\dvol,m)$ with $m<\infty$ such that $\Ric_\phi^m=\lambda g$ for some constant $\lambda<0$ which are conformally compact in the sense of conformal equivalence of smooth metric measure spaces as defined in~\cite{Case2010a}.  More precisely, we will prove the following result.

\begin{thm}
\label{thm:connectedness}
Let $(M^n,g,e^{-\phi}\dvol,m)$ be a conformally compact quasi-Einstein smooth metric measure space with $m<\infty$.  Suppose that the first eigenvalue of the weighted conformal Laplacian of the conformal boundary $[\partial M,\og\rv_{T\partial M},\ov^m\dvol]$ is nonnegative.  If $m>0$, suppose additionally that $(M^n,g,e^{-\phi}\dvol,m)$ has nonnegative characteristic constant.  Then $\partial M$ is connected.
\end{thm}

For an explanation of the terminology used in this result, see Section~\ref{sec:smms} and Section~\ref{sec:ah}.  We expect that one can make sense of Theorem~\ref{thm:connectedness} in the limiting case $m=\infty$, however it is not clear to us how to define conformally compact quasi-Einstein smooth metric measure spaces with $m=\infty$, and in particular, what should be the ``conformal infinity.''

The proof of Theorem~\ref{thm:connectedness} is based upon constructing a family of hypersurfaces tending to infinity which satisfy the hypotheses of Theorem~\ref{thm:splitting}, and closely parallels the proof given by Cai and Galloway~\cite{CaiGalloway1999} to establish the case $m=0$.  Since only one of the ends of~\eqref{eqn:splitting_conclusion} is conformally compact, Theorem~\ref{thm:connectedness} then follows.

Like Theorem~\ref{thm:splitting}, Theorem~\ref{thm:connectedness} can also be regarded as a consequence of Theorem~\ref{thm:weighted_cheng}.  More precisely, we will also prove Theorem~\ref{thm:ccqe_to_spectrum}, which states, as a generalization to smooth metric measure spaces of a result of Lee~\cite{Lee1995}, that conformally compact quasi-Einstein smooth metric measure spaces as in Theorem~\ref{thm:connectedness} achieve equality in the weighted Cheng estimate of Theorem~\ref{thm:weighted_cheng}.

Our approach to Theorem~\ref{thm:splitting} suggests possible interpretations of the two splitting theorems of Munteanu and Wang~\cite{MunteanuWang2011,MunteanuWang2011b} as consequences of some useful notion of ``conformally compact gradient Ricci solitons,'' the formalization of which we wish to highlight as an interesting open problem.  We hope that this idea of studying conformally compact gradient Ricci solitons might be useful in obtaining a better understanding of the structure of gradient Ricci solitons.

This article is organized as follows: In Section~\ref{sec:smms} we recall some facts and definitions pertaining to smooth metric measure spaces necessary for the proof of Theorem~\ref{thm:splitting}.  In Section~\ref{sec:sharp} we describe three families of quasi-Einstein smooth metric measure spaces and use them to establish the sharpness of Theorem~\ref{thm:splitting}, as well as the necessity of the scalar curvature assumption in~\cite[Theorem~1.4]{MunteanuWang2011b}.  In Section~\ref{sec:splitting} we prove Theorem~\ref{thm:splitting} through an adaptation of the argument of Cai and Galloway, and then apply the proof to realize Theorem~\ref{thm:splitting} as a special case of Theorem~\ref{thm:weighted_cheng}.  In Section~\ref{sec:ah} we make precise the notion of a conformally compact quasi-Einstein smooth metric measure space in the case $m<\infty$ and prove both Theorem~\ref{thm:connectedness} and the Lee-type result on the bottom of the spectrum of the weighted Laplacian of such a space.

\subsection*{Acknowledgments}
The first author would like to thank Ovidiu Munteanu and Guofang Wei for valuable conversations while this work was taking shape.

\section{Smooth metric measure spaces}
\label{sec:smms}

To begin, let us recall some notions important to the study of smooth metric measure spaces.  We follow~\cite{Case2010a} for our conventions, and in particular adapt our definition of the weighted mean curvature of a hypersurface to those conventions and the Riemannian conventions used by Cai and Galloway~\cite{CaiGalloway1999}.

\begin{defn}
A \emph{smooth metric measure space} is a four-tuple $(M^n,g,e^{-\phi}\dvol,m)$ of a Riemannian manifold $(M^n,g)$, a smooth measure $e^{-\phi}\dvol$ for $\phi\in C^\infty(M)$ and $\dvol$ the Riemannian volume element associated to $g$, and a dimensional parameter $m\in[0,\infty]$.
\end{defn}

We will sometimes denote smooth metric measure spaces more succinctly as triples $(M^n,g,v^m\dvol)$, where the notation $v^m\dvol$ for the measure encodes the dimensional parameter $m$ as the exponent of a function $0<v\in C^\infty(M)$.  When we do this, we will always formally define $v^m=e^{-\phi}$, with the symbol $v^\infty\dvol$ either interpreted to denote a measure $e^{-\phi}\dvol$ for some $\phi\in C^\infty(M)$ being regarded as an ``infinite-dimensional'' measure, or interpreted to denote that one should take a limit to determine the measure in the case when $v$ is defined for all $m$ at once (cf.\ Section~\ref{sec:sharp}).

The most basic object of interest on a smooth metric measure space is the \emph{weighted Laplacian} $\Delta_\phi=\Delta-\nabla\phi$.  This operator is the natural modification of the Laplacian to this setting, in the sense that it is equivalently defined by $\Delta_\phi=-\nabla^\ast\nabla$ for $\nabla^\ast$ the formal adjoint of $\nabla$ with respect to the measure $e^{-\phi}\dvol$.

The role of the dimensional parameter $m$ is to specify that geometric invariants associated to smooth metric measure spaces should be made as if the measure were that of an $(m+n)$-dimensional Riemannian manifold; see~\cite{Case2010a,Wei_Wylie} for a more precise description of this heuristic.  For us, the most important such invariants are the Bakry-\'Emery Ricci curvature and the weighted scalar curvature.

\begin{defn}
The \emph{Bakry-\'Emery Ricci tensor $\Ric_\phi^m$} and the \emph{weighted scalar curvature $R_\phi^m$} of a smooth metric measure space $(M^n,g,v^m\dvol)$ are defined by
\begin{align*}
\Ric_\phi^m & := \Ric + \nabla^2\phi - \frac{1}{m}d\phi\otimes d\phi = \Ric - \frac{m}{v}\nabla^2 v \\
R_\phi^m & := R + 2\Delta\phi - \frac{m+1}{m}\lv\nabla\phi\rv^2 = R - \frac{2m}{v}\Delta v - \frac{m(m-1)}{v^2}\lv\nabla v\rv^2 .
\end{align*}
\end{defn}

The relation between the Bakry-\'Emery Ricci tensor and the dimensional interpretation of the measure $v^m\dvol$ is made through the \emph{weighted Bochner formula}
\begin{equation}
\label{eqn:weighted_bochner}
\frac{1}{2}\Delta_\phi\lv\nabla u\rv^2 = \lv\nabla^2 u\rv^2 + \lp\nabla\Delta_\phi u,\nabla u\rp + \Ric_\phi^m(\nabla u,\nabla u) + \frac{1}{m}\lp\nabla\phi,\nabla u\rp^2 .
\end{equation}
More precisely, the Cauchy-Schwarz inequality yields from~\eqref{eqn:weighted_bochner} the \emph{weighted Bochner inequality}
\[ \frac{1}{2}\Delta_\phi\lv\nabla u\rv^2 \geq \frac{1}{m+n}\left(\Delta_\phi u\right)^2 + \lp\nabla u,\nabla\Delta_\phi u\rp + \Ric_\phi^m(\nabla u,\nabla u) \]
for all $u\in C^3(M)$, so that dimensions appear here in the same way as they do in the usual Bochner inequality for $(m+n)$-dimensional Riemannian manifolds.  In particular, this leads in the usual way to the following comparison theorem for the weighted Laplacian, proven by Qian~\cite{Qian1997} for $m<\infty$ and Wei and Wylie~\cite{Wei_Wylie} for $m=\infty$.

\begin{prop}
\label{prop:laplacian_comparison}
Let $(M^n,g,v^m\dvol)$ be a smooth metric measure space with $\Ric_\phi^m\geq\lambda g$.  Fix $p\in M$, and let $r(x)=d(p,x)$ denote the distance from $p$.
\begin{enumerate}
\item If $m<\infty$, then
\[ \Delta_\phi r \leq (m+n-1)\frac{\cn_\lambda(r)}{\sn_\lambda(r)} , \]
where $\sn_\lambda(r)$ is the unique solution to $u^{\prime\prime}+\frac{\lambda}{m+n-1} u=0$, $u(0)=0$, $u^\prime(0)=1$, and $\cn_\lambda(r)=\sn_\lambda^\prime(r)$.
\item If $m=\infty$, $\lambda=0$, and $\lv\nabla\phi\rv^2\leq a^2$, then
\[ \Delta_\phi r \leq \frac{n-1}{r} + a . \]
\end{enumerate}
\end{prop}

In this article, we will discuss quasi-Einstein smooth metric measure spaces in the following sense.

\begin{defn}
A smooth metric measure space $(M^n,g,v^m\dvol)$ is \emph{quasi-Einstein} if there is a constant $\lambda\in\bR$ such that $\Ric_\phi^m=\lambda g$.  In this case, we call $\lambda$ the \emph{quasi-Einstein constant} of $(M^n,g,v^m\dvol)$.

The \emph{characteristic constant $\mu\in\bR$} of a quasi-Einstein smooth metric measure space with quasi-Einstein constant $\lambda$ is the constant such that
\begin{equation}
\label{eqn:charconst}
R_\phi^m + m\mu v^{-2} = (m+n)\lambda
\end{equation}
holds; when $m=\infty$, the characteristic constant is $\mu=\lambda$.
\end{defn}

It is a result of Kim and Kim~\cite{Kim_Kim} that every connected quasi-Einstein smooth metric measure space has a well-defined characteristic constant.

Lastly, we must define the weighted mean curvature of a hypersurface in a smooth metric measure space.

\begin{defn}
Let $(M^n,g,v^m\dvol)$ be a smooth metric measure space and let $\Sigma^{n-1}\subset M$ be a smooth hypersurface with outward pointing normal $\eta$.  The \emph{weighted mean curvature $H_\phi$} is the function
\[ H_\phi = \delta_\phi\eta \]
for $\delta_\phi=-\nabla^\ast$ the \emph{weighted divergence}.
\end{defn}

There is a weighted Gauss equation which relates the weighted mean curvature and weighted scalar curvature of a hypersurface to the weighted scalar curvature and Bakry-\'Emery Ricci tensor of the ambient smooth metric measure space.

\begin{prop}
Let $(M^n,g,v^m\dvol)$ be a smooth metric measure space and let $\Sigma^{n-1}\subset M$ be a smooth hypersurface with outward pointing normal $\eta$.  It holds that
\begin{equation}
\label{eqn:weighted_gauss}
R_\phi^m = \widehat{R_\phi^m} + 2\Ric_\phi^m(\eta,\eta) + \lv B\rv^2 + \frac{1}{m}\lp\nabla\phi,\eta\rp^2 - (H_\phi)^2 ,
\end{equation}
where $\widehat{R_\phi^m}$ is the weighted scalar curvature of $(\Sigma,g\rv_{T\Sigma},v^m\dvol)$ and $B$ is the second fundamental form of $\Sigma\subset M$.
\end{prop}

\begin{proof}

This follows immediately from the usual Gauss equation and the formulae
\begin{align*}
\lv\nabla f\rv_g^2 & = \lv\hat\nabla f\rv_{\hat g}^2 + \lp\nabla f,\eta\rp^2 \\
\Delta f & = \widehat{\Delta}f + \nabla^2f(\eta,\eta) + H\lp\nabla f,\eta\rp
\end{align*}
relating derivatives on $(M^n,g)$ to derivatives on $(\Sigma,\hat g=g\rv_{T\Sigma})$.
\end{proof}
\section{The sharpness of Theorem~\ref{thm:splitting}}
\label{sec:sharp}

In order to establish the sharpness of Theorem~\ref{thm:splitting}, we consider three different families of smooth metric measure spaces.  The first example is the family~\eqref{eqn:splitting_conclusion}, which we briefly discuss in order to verify that it does indeed satisfy the hypotheses of Theorem~\ref{thm:splitting}.  The other two examples both satisfy all of the hypotheses of Theorem~\ref{thm:splitting} save~\eqref{eqn:splitting_mean_curvature}: For these examples, the limit is finite but nonzero.

\begin{example}
\label{ex:hyperbolic_cusp}
Fix $n\geq 2$, $m\in[0,\infty]$, and let $(N^{n-1},h)$ be a complete Riemannian manifold with nonnegative Ricci curvature.  The smooth metric measure space
\begin{equation}
\label{eqn:negative_cusp}
\left( \bR\times N^{n-1}, dt^2\oplus e^{\frac{2t}{m+n-1}}h, \left(e^{\frac{t}{m+n-1}}\right)^m\dvol\right)
\end{equation}
is such that $\Ric_\phi^m\geq-\frac{1}{m+n-1}g$ and the weighted mean curvature $H_\phi$ of the hypersurface $\{(t,p)\colon p\in N\}$ is $H_\phi=1$.
\end{example}

The verification that $\Ric_\phi^m\geq-\frac{1}{m+n-1}g$ and $H_\phi=1$ follow easily from the formulae for the Ricci curvature and Hessian of the lift of a function on the base of a warped product, together with the fact that $H_\phi=\Delta_\phi t$, which can be found in~\cite[Chapter~9]{Besse}.  A key point about this example is that our normalization ensures that it makes sense in the limit $m=\infty$, where it is the smooth metric measure space
\[ \left(\bR\times N^{n-1}, dt^2\oplus h, e^t\dvol, \infty\right) \]
with $\Ric_\phi^\infty\geq0$ and for which the hypersurface $\{(t,p)\colon p\in N\}$ has weighted mean curvature $H_\phi=1$.  The next two examples will also be normalized so that they make sense in the limit $m=\infty$.

If one assumes additionally that $(N,h)$ is Ricci flat (i.e.\ equality holds in the specified curvature bound), then~\eqref{eqn:negative_cusp} is quasi-Einstein (i.e.\ equality holds in the lower bound on $\Ric_\phi^m$).  This behavior will also persist in the next two examples.

\begin{example}
\label{ex:negative_elliptic_gaussian_warped_cylinder}
Fix $n\geq 3$, $m\in[0,\infty]$, set $k=\sqrt{m+n-1}$, and let $(N^{n-1},h)$ be a complete Riemannian manifold with $\Ric_h\geq-\frac{m+n-2}{m+n-1}h$.  The smooth metric measure space
\begin{equation}
\label{eqn:negative_elliptic_gaussian_warped_cylinder}
\left( \bR\times N^{n-1}, dt^2\oplus\cosh^2\left(\frac{t}{k}\right)h, \left(\cosh\left(\frac{t}{k}\right)\right)^m\dvol \right)
\end{equation}
is such that $\Ric_\phi^m\geq-g$ and the weighted mean curvature $H_\phi$ of the hypersurface $\{(t,p)\colon p\in N\}$ is
\begin{equation}
\label{eqn:negative_elliptic_gaussian_warped_cylinder_mean_curvature}
\begin{split}
H_\phi & = \sqrt{m+n-1}\tanh\left(\frac{t}{\sqrt{m+n-1}}\right) \\
& = \sqrt{m+n-1} - 2\sqrt{m+n-1}e^{-\frac{2t}{\sqrt{m+n-1}}} + o\left(e^{-\frac{2t}{\sqrt{m+n-1}}}\right)
\end{split}
\end{equation}
for $t$ large.
\end{example}

Using the terminology of~\cite{Case2010a,HePetersenWylie2010}, this example is obtained from the one-dimensional negative elliptic Gaussian by taking a warped product.  In the case $m=\infty$, this is the product
\[ \left(\bR\times N^{n-1}, dt^2\oplus h, e^{t^2/2}\dvol \right) \]
of the one-dimensional expanding Gaussian with a Riemannian manifold with Ricci curvature bounded below by $-1$.  In particular, Munteanu and Wang~\cite{MunteanuWang2011b} showed that the only complete expanding gradient Ricci solitons with $\Ric_\phi^\infty=-g$ and $R\geq-(n-1)$ are of this form.

The verification of the first line of~\eqref{eqn:negative_elliptic_gaussian_warped_cylinder_mean_curvature} again follows easily from the formula for the Laplacian of a warped product, while the second line uses elementary properties of the hyperbolic tangent function.  In particular, \eqref{eqn:negative_elliptic_gaussian_warped_cylinder_mean_curvature} reveals that for $h_t$ defined in terms of $\Sigma_t$ as in Theorem~\ref{thm:splitting},
\[ \lim_{t\to\infty} \left(\sqrt{m+n-1}-h_t\right)e^{\frac{2t}{\sqrt{m+n-1}}} = 2\sqrt{m+n-1} ; \]
i.e.\ Example~\ref{ex:negative_elliptic_gaussian_warped_cylinder} establishes the sharpness of Theorem~\ref{thm:splitting}, as easily follows by rescaling so that $\Ric_\phi^m\geq-\frac{1}{m+n-1}g$.

\begin{example}
\label{ex:besse}
Fix $n\geq 4$, $m\in[0,\infty]$, and $(N^{n-2},h)$ a complete Riemannian manifold satisfying $\Ric_h\geq-\frac{m+n-3}{m+n-1}h$.  Let $v\colon\bR\to(0,\infty)$ be the unique solution to the ODE
\begin{equation}
\label{eqn:besse_ode}
\left(\frac{dv}{dt}\right)^2 = \frac{v^2-1}{m+n-1} + \frac{2(m+n-3)^{(m+n-3)/2}}{(m+n-1)^{(m+n+1)/2}}v^{3-m-n}
\end{equation}
such that $v(0)=1$ and $v^\prime>0$.  Then the smooth metric measure space
\begin{equation}
\label{eqn:besse}
\left( \bR\times S^1\times N^{n-2}, dt^2\oplus\left((m+n-1)v^\prime(t)\right)^2d\theta^2\oplus \left(v(t)\right)^2 h, \left(v(t)\right)^m\dvol\right)
\end{equation}
is such that $\Ric_\phi^m\geq-g$ and the weighted mean curvature $H_\phi$ of the level set $\Sigma_t=\{(t,p)\colon p\in S^1\times N^{n-2}\}$ is
\begin{equation}
\label{eqn:besse_mean_curvature}
\begin{split}
H_\phi & = \frac{v^{\prime\prime}(t)}{v^\prime(t)} + (m+n-2)\frac{v^\prime(t)}{v(t)} \\
& = \sqrt{m+n-1} - \frac{m+n-3}{2\sqrt{m+n-1}}v^{-2} + o(v^{-2})
\end{split}
\end{equation}
for $t$ large.
\end{example}

Again using the terminology of~\cite{Case2010a}, this is the warped product of~\cite[Example~9.118(c)]{Besse} with $(N^{n-2},h)$.  In particular, we point out that this corrects an error in~\cite{Besse}, where the factors of one-half in the exponent of the last summand of~\eqref{eqn:besse_ode} are missing.  These factors are necessary to ensure the completeness of the metric of~\eqref{eqn:besse}, which requires that $\int_a^0\frac{dv}{\sqrt{F(v)}}=\infty$ for $a$ the infimum of $v$ and $F(v)$ the right hand side of~\eqref{eqn:besse_ode}: Writing
\[ F(v) = \frac{v^2-1}{m+n-1} + 2cv^{3-m-n}, \]
it is readily verified that $c=\frac{(m+n-3)^{(m+n-3)/2}}{(m+n-1)^{(m+n+1)/2}}$ is the only value of $c$ for which this integral is infinite (cf.\ \cite[Appendix~A]{HePetersenWylie2010}).

In the limiting case $m=\infty$, Example~\ref{ex:besse} is the smooth metric measure space
\begin{equation}
\label{eqn:besse_infty}
\left( \bR\times S^1\times N^{n-2}, dt^2\oplus\left(\phi^\prime(t)\right)^2d\theta^2\oplus h, e^{-\phi(t)}\dvol \right)
\end{equation}
for $\phi\colon\bR\to(-\infty,1)$ the unique solution to the ODE
\begin{equation}
\label{eqn:besse_infty_ode}
\left(\frac{d\phi}{dt}\right)^2 = 2e^{\phi-1} - 2\phi
\end{equation}
with $\phi(0)=0$ and $\phi^\prime<0$.

% \begin{remark}
% This example can also be read off from the recent work of Bernstein--Mettler and Daniel Ramos.  We point out that their list of complete nontrivial two-dimensional gradient Ricci solitons can be read off from the classification of Einstein warped products over two-dimensional bases in~\cite[Section~9.118]{Besse} by taking the limit as the dimension of the fibers becomes infinite, as we did above.
% \end{remark}

The verification of the first line of~\eqref{eqn:besse_mean_curvature} is again a simple application of the formula for the Laplacian of a warped product.  To verify the second line, observe that
\[ \lim_{t\to\infty}\frac{v^\prime(t)}{v(t)} = \frac{1}{\sqrt{m+n-1}} . \]
The second line of~\eqref{eqn:besse_mean_curvature} follows from this and differentiation of~\eqref{eqn:besse_ode}.  Moreover, as $v\sim e^{\frac{t}{\sqrt{m+n-1}}}$ for $t$ large --- that is, there is a constant $c>0$ such that $ve^{-\frac{t}{\sqrt{m+n-1}}}\to c\not=0$ as $t\to\infty$ --- it follows from~\eqref{eqn:besse_mean_curvature} that
\[ \lim_{t\to\infty}\left(\sqrt{m+n-1}-h_t\right)e^{\frac{2t}{\sqrt{m+n-1}}} = A \]
for some $A\in(0,\infty)$.  In particular, Example~\ref{ex:besse} also establishes the sharpness of Theorem~\ref{thm:splitting}.

Note also that in the case $m=\infty$ with $(N^{n-2},h)$ such that $\Ric_h=-g$, the scalar curvature of~\eqref{eqn:besse_infty} is
\[ R = -\frac{2\phi^{\prime\prime}}{\phi^\prime} + 2 - n = -\frac{2(e^{\phi-1}-1)}{\phi^\prime} + 2 - n . \]
It follows from~\eqref{eqn:besse_infty_ode} that $\phi^\prime\sim\phi-1$ as $\phi\to1$ (i.e.\ as $t\to-\infty)$, and hence we see that $R\to-n$ as $t\to-\infty$.  In particular, if $N$ is further taken to be compact, then~\eqref{eqn:besse_infty} is a complete expanding gradient Ricci soliton with two ends and $\Ric_\phi^\infty=-g$, and therefore provides an example showing that the assumption $R\geq-(n-1)$ cannot be relaxed to $R>-n$ in~\cite[Theorem~1.4]{MunteanuWang2011b}.

Finally, we point out that in each example the end $t\to\infty$ is conformally compact in the sense of Definition~\ref{defn:ccqe}, while the ends $t\to-\infty$ in Example~\ref{ex:negative_elliptic_gaussian_warped_cylinder} and Example~\ref{ex:besse} both behave analogously to hyperbolic cusps.
\section{The proof of Theorem~\ref{thm:splitting}}
\label{sec:splitting}

Let us now turn to the proof of Theorem~\ref{thm:splitting}.  First, let us motivate both Theorem~\ref{thm:splitting} and its proof by demonstrating in what sense Theorem~\ref{thm:splitting} can be regarded as a rigidity result.

\begin{prop}
\label{prop:strong_cg}
Let $(M^n,g,v^m\dvol)$ be a complete smooth metric measure space with compact boundary $\Sigma$ such that $\Ric_\phi^m\geq-\frac{1}{m+n-1}g$ and the weighted mean curvature $H_\phi$ of $\Sigma$ satisfies $H_\phi>1$.  If $m=\infty$, assume additionally that $\lv\nabla\phi\rv^2\leq 1$.  Then $M$ is compact.
\end{prop}

In particular, note that Theorem~\ref{thm:splitting} characterizes those smooth metric measure spaces with multiple ends which arise by relaxing the assumption $H_\phi>1$ in Proposition~\ref{prop:strong_cg} to the assumption~\eqref{eqn:splitting_mean_curvature}.

An important step in the proof of Proposition~\ref{prop:strong_cg} is to show that the statement when $m=\infty$ reduces to the statement when $m<\infty$.  We do this via the following general lemma.

\begin{lem}
\label{lem:approximation}
Let $(M^n,g,e^{-\phi}\dvol,\infty)$ be a complete smooth metric measure space with $\Ric_\phi^\infty\geq0$ and $\lv\nabla\phi\rv^2\leq1$.  Then for any $m^\prime\in(0,\infty)$, the smooth metric measure space
\begin{equation}
\label{eqn:approximation}
\left( M^n,\og,\ov^{m^\prime}\dvol_{\og} \right) := \left( M^n,\frac{m^\prime+n-1}{m^\prime}g, e^{-\phi}\dvol_g, m^\prime \right)
\end{equation}
satisfies $\overline{\Ric_\phi^{m^\prime}}\geq-\frac{1}{m^\prime+n-1}\og$.  Moreover, if $\Sigma\subset M$ has weighted mean curvature $H_\phi$ measured with respect to $(M^n,g,e^{-\phi}\dvol,\infty)$, then its weighted mean curvature $\overline{H_\phi}$ measured with respect to~\eqref{eqn:approximation} satisfies $\overline{H_\phi}=\sqrt{\frac{m^\prime}{m^\prime+n-1}}H_\phi$.
\end{lem}

\begin{proof}

The formula relating $\overline{H_\phi}$ and $H_\phi$ follows by the scaling of the weighted mean curvature.  The lower bound for $\overline{\Ric_\phi^{m^\prime}}$ follows from the fact
\[ \overline{\Ric_\phi^{m^\prime}} = \Ric_\phi^\infty - \frac{1}{m^\prime}d\phi\otimes d\phi \geq -\frac{1}{m^\prime}g, \]
where the inequality follows from the assumption $\lv\nabla\phi\rv_g^2\leq 1$.
\end{proof}

\begin{proof}[Proof of Proposition~\ref{prop:strong_cg}]

First, we observe that it suffices to consider the case $m<\infty$.  Indeed, if $m=\infty$, define $\delta=\inf_\Sigma H_\phi-1>0$.  We may thus choose $m^\prime<\infty$ sufficiently large so that $\sqrt{\frac{m^\prime}{m^\prime+n-1}}(1+\delta)>1$.  It follows from Lemma~\ref{lem:approximation} that the smooth metric measure space~\eqref{eqn:approximation} satisfies the hypotheses of Proposition~\ref{prop:strong_cg}.

To establish the proposition in the case $m<\infty$, fix $q\in M$, set $d=d(q,\partial M)$, and choose $p\in\partial M$ such that $d=d(p,q)$.  Let $\sigma\colon[0,d]\to M$ be a unit speed minimizing geodesic with $\sigma(0)=p$ and $\sigma(d)=q$.  Denote by $\rho$ the distance function from $\partial M$, $\rho(x)=d(x,\partial M)$.  Then $\rho$ is smooth when restricted to $\sigma\big([0,d)\big)$.  Denote by $H_\phi(s)=H_\phi(\sigma(s))$ the weighted mean curvature of the level set $\rho^{-1}(s)$ at $\sigma(s)$.  For $s\in[0,d)$, this is well-defined, and moreover, we have that $H_\phi(s)=-\Delta_\phi\rho$.  By the Bochner identity~\eqref{eqn:weighted_bochner} and the Cauchy-Schwarz inequality, we then see that
\begin{align*}
H_\phi^\prime(s) & = \lv\nabla^2\rho\rv^2 + \Ric_\phi^m(\nabla\rho,\nabla\rho) + \frac{1}{m}\lp\nabla\phi,\nabla\rho\rp^2 \\
& \geq \frac{1}{m+n-1}\left(H_\phi^2 - 1\right) .
\end{align*}
Let $\delta=H_\phi(p)-1>0$.  By comparison with the solution to the ODE
\[ \begin{cases}
   F^\prime & = \frac{1}{m+n-1}\left(F^2-1\right) \\
   F(0) & = 1+\delta,
   \end{cases} \]
namely
\[ F(t) = \frac{2+\delta+\delta e^{\frac{2t}{m+n-1}}}{2+\delta-\delta e^{\frac{2t}{m+n-1}}} , \]
we see that since $H_\phi$ is smooth, it must hold that $d\leq\frac{m+n-1}{2}\ln\frac{2+\delta}{\delta}$.  In particular, $M$ is compact.
\end{proof}

Let us now turn to the proof of Theorem~\ref{thm:splitting}.  As is typical in generalizing comparison results from Riemannian geometry to smooth metric measure spaces, the proof of Theorem~\ref{thm:splitting} follows from the corresponding proof in the Riemannian setting~\cite[Theorem~3]{CaiGalloway1999} with only minor modifications to take into account the role of the measure.  The main subtlety is to ensure that the result remains true in the limit $m=\infty$ (cf.\ \cite{Wei_Wylie}), and for this reason we find it useful to provide here the details of the proof of Theorem~\ref{thm:splitting}.  To that end, we recall the method of generalized support functions employed by Cai and Galloway~\cite{CaiGalloway1999}.

\begin{defn}
Let $(M^n,g,v^m\dvol)$ be a smooth metric measure space and let $f\in C^0(M)$.  A \emph{lower support for $f$ at $p\in M$} is a function $u$ defined on a neighborhood $U$ of $p$ which is $C^2$ and such that $u\leq f$ in $U$ with $u(p)=f(p)$.

We say that $\Delta_\phi f\geq a$ \emph{in the support sense} if for each $p\in M$ and each $\varepsilon>0$, there exists a lower support $u_{p,\varepsilon}$ for $f$ at $p$ such that $\Delta_\phi u_{p,\varepsilon}\geq a-\varepsilon$.

We say that $\Delta_\phi f\geq a$ \emph{in the generalized support sense} if for each $p\in M$, there exists a neighborhood $U$ of $p$ on which there is a sequence $\{f_k\}\subset C^0(U)$ such that $f_k\to f$ uniformly in $U$ and $\Delta_\phi f_k\geq a_k$ in the support sense for $a_k\to a$.
\end{defn}

The following maximum principle, which we state without proof, follows exactly as in~\cite{CaiGalloway1999,EschenburgHeintze1984}.

\begin{lem}
Let $(M^n,g,v^m\dvol)$ be a connected smooth metric measure space and suppose that $f\in C^0(M)$ satisfies $\Delta_\phi f\geq 0$ in the generalized support sense.  If $f$ attains its maximum, then it is constant.
\end{lem}

\begin{proof}[Proof of Theorem~\ref{thm:splitting}]

To begin, suppose that $M$ has at least two ends, so that there is a compact set $K\subset M$ and there are two connected components $E_1,E_2\subset M\setminus K$.  Since $M\setminus K$ has finitely many components, we may assume that $\Sigma_k\subset E_1$ for all $k$.  One can then construct a line $\sigma\colon\bR\to M$ by taking a sequence of points $p_k\in E_2$ tending to infinity and $q_k\in \Sigma_k$ with $d(p_k,q_k)=d(p_k,\Sigma_k)$, connecting them via a minimizing geodesic $\sigma_k\colon[-a_k,b_k]\to M$ going from $p_k$ to $q_k$ --- that is, $\sigma_k(-a_k)=p_k$ and $\sigma_k(b_k)=q_k$ --- and considering the limit $k\to\infty$ (cf.\ \cite{CaiGalloway1999}).  Necessarily $\sigma$ will intersect $K$, and we may parametrize $\sigma$ so that $\sigma(0)=\bar o\in K$.  Associated to the two ends of $\sigma$ are two Busemann functions, one to the sequence $\{\Sigma_k\}$ and one to the ray $\sigma_{(-\infty,0]}$.

First, for each $k$, define $\beta_k\colon M\to\bR$ by
\[ \beta_k(x) = d(\bar o,\Sigma_k) - d(x,\Sigma_k) . \]
Passing to a subsequence as necessary, it is easy to conclude that $\beta_k$ converges on compact sets to a continuous function $\beta\colon M\to\bR$, the \emph{Busemann function associated with $\{\Sigma_k\}$} (cf.\ \cite{CaiGalloway1999}).  Our goal is to show that $\Delta_\phi\beta\geq 1$ in the generalized support sense.  To that end, fix $q\in M$, let $B=B(q,r)$ be a small geodesic ball around $q$, and let $k$ be large enough that $B$ is on the inside of $\Sigma_k$.  We will show that there are constants $c_k$ such that $c_k\to1$ as $k\to\infty$ and $\Delta_\phi\beta_k\rv_B\geq c_k$ in the support sense.

To show that $\Delta_\phi\beta_k\rv_B\geq c_k$ in the support sense, and also identify the constants $c_k$, fix $p\in B$ and $\varepsilon>0$, and choose $z\in\Sigma_k$ such that $d(p,z)=d(p,\Sigma_k)$.  Let $V\subset\Sigma_k$ be a neighborhood of $z$.  By bending $V$ slightly to the outside of $\Sigma_k$, we can deform $V$ to a smooth hypersurface $V^\prime\subset M$ such that (1) $z\in V^\prime$ is the unique closest point to $p$ in $V^\prime$, (2) the second fundamental form of $V^\prime$, measured with respect to the outward pointing normal, is strictly less than that of $V$, and (3) the weighted mean curvature $H_\phi^{V^\prime}(z)$ of $V^\prime$ at $z$ satisfies $H_\phi^{V^\prime}(z)\geq H_\phi^V(z)-\varepsilon\geq h_k-\varepsilon$.  Let $\gamma\colon[0,l]\to M$ be a unit speed geodesic such that $\gamma(0)=z$ and $\gamma(l)=p$.  Hence, by the construction of $V^\prime$, the function
\[ \beta_k^{p,\varepsilon}(x) := d(\bar o,\Sigma_k) - d(x,V^\prime) \]
is a lower support for $\beta_k\rv_B$ at $p$.

Now, for $s\in[0,l]$, define $H(s)=\Delta_\phi\beta_k^{p,\varepsilon}\left(\gamma(s)\right)$.

\emph{Case 1}: $m<\infty$.  As in the proof of Proposition~\ref{prop:strong_cg}, it holds that
\[ \frac{d}{ds}H(s) \geq \frac{1}{m+n-1}\left(H^2(s)-1\right) . \]
Since $H(0)\geq h_k-\varepsilon$ and $h_k-\varepsilon<1$, it follows by comparison to the ODE
\[ \begin{cases}
     F^\prime & = \frac{1}{m+n-1}\left(F^2-1\right) \\
     F(0) & = h_k-\varepsilon,
  \end{cases} \]
that
\[ \Delta_\phi\beta_k^{p,\varepsilon}(p) = H(l) \geq \frac{1+h_k-\varepsilon-(1-(h_k-\varepsilon))e^{\frac{2l}{m+n-1}}}{1+h_k-\varepsilon+(1-(h_k-\varepsilon))e^{\frac{2l}{m+n-1}}} . \]
By the triangle inequality, $e^{\frac{2l}{m+n-1}}\leq Ce^{\frac{2d(o,\Sigma_k)}{m+n-1}}$, where $C=e^\frac{2d(o,x)}{m+n-1}$, whence
\[ \Delta_\phi\beta_k^{p,\varepsilon}(p) \geq \frac{1+h_k-\varepsilon-C(1-(h_k-\varepsilon))e^{\frac{2d(o,\Sigma_k)}{m+n-1}}}{1+h_k-\varepsilon+C(1-(h_k-\varepsilon))e^{\frac{2d(o,\Sigma_k)}{m+n-1}}} . \]
This shows that at $p$, $\beta_k$ satisfies
\[ \Delta_\phi\beta_k\geq\frac{1+h_k-C(1-h_k)e^{\frac{2d(o,\Sigma_k)}{m+n-1}}}{1+h_k+C(1-h_k)e^{\frac{2d(o,\Sigma_k)}{m+n-1}}} \]
in the support sense.

\emph{Case 2}: $m=\infty$.  Using the assumption $\lv\nabla\phi\rv^2\leq1$, the Bochner inequality easily yields
\[ \frac{d}{ds}H(s) \geq \frac{1}{n+a-1}H^2(s) - \frac{1}{a} \]
for any $a>0$.  Since $h_k-\varepsilon<1<\sqrt{\frac{n+a-1}{a}}$, the analogous comparison to the above yields
\[ \sqrt{\frac{a}{n+a-1}} \Delta_\phi\beta_k^{p,\varepsilon}(p) \geq \frac{ \sqrt{\frac{n+a-1}{a}}+h_k-\varepsilon - \left(\sqrt{\frac{n+a-1}{a}}-(h_k-\varepsilon)\right) e^{2l/\sqrt{a(n+a-1)}} }{ \sqrt{\frac{n+a-1}{a}} + h_k-\varepsilon + \left(\sqrt{\frac{n+a-1}{a}} - (h_k-\varepsilon)\right)e^{2l/\sqrt{a(n+a-1)}} } . \]
In particular, taking $a=l$ and letting $\varepsilon\to0$, it follows that at $p$
\[ \Delta_\phi\beta_k\geq\sqrt{\frac{n+l-1}{l}}\left(\frac{ \sqrt{\frac{n+l-1}{l}}+h_k-\left(\sqrt{\frac{n+l-1}{l}}-h_k\right) e^{2\sqrt{\frac{l}{n+l-1}}} }{ \sqrt{\frac{n+l-1}{l}} + h_k + \left(\sqrt{\frac{n+l-1}{l}} - h_k\right)e^{2\sqrt{\frac{l}{n+l-1}}} }\right) \]
in the support sense.

Hence, in either case the assumption~\eqref{eqn:splitting_mean_curvature} implies, by taking $k\to\infty$, that $\Delta_\phi\beta\geq1$ in the generalized support sense, as desired.

Second, for each $s\in\bR$, define $b_s\colon M\to\bR$ by
\[ b_s(x) = d\big(\bar o,\sigma(s)\big) - d\big(x,\sigma(s)\big) = \lv s\rv - d\big(x,\sigma(s)\big) . \]
Passing to a subsequence as necessary, one concludes that $b_s$ converges on compact sets as $s\to-\infty$ to a continuous function $b\colon M\to\bR$, the \emph{Busemann function associated with $\sigma\rv_{(-\infty,0]}$}.  From Proposition~\ref{prop:laplacian_comparison} it follows that $\Delta_\phi b\geq-1$.

Now, set $F=\beta+b$.  By the above, we have that $\Delta_\phi F\geq 0$.  By the triangle inequality, it is easy to check that $F\leq 0$ (cf.\ \cite{CaiGalloway1999}).  Moreover, by construction we have that $F(\bar o)=0$.  Hence, by the maximum principle, $F\equiv 0$, whence $\beta=-b$.  This implies that $\Delta_\phi\beta=1$, which, by which elliptic regularity, implies that $\beta$ is smooth.  Since $\beta$ is a Busemann function, this implies that $\lv\nabla\beta\rv^2=1$.  By the Bochner formula, it thus holds that
\[ 0 = \lv\nabla^2\beta\rv^2 + \Ric_\phi^m(\nabla\beta,\nabla\beta) + \frac{1}{m}\lp\nabla\beta,\nabla\phi\rp^2 \geq \frac{1}{m+n-1}\left((\Delta_\phi\beta)^2-1\right) = 0 . \]
Hence equality holds in all steps; i.e.\ $\Ric_\phi^m(\nabla\beta,\nabla\beta)=-\frac{1}{m+n-1}$, $\nabla^2\beta\big|_{\nabla\beta^\perp}=\frac{\Delta\beta}{n-1}g\big|_{\nabla\beta^\perp}$, and $\Delta\beta=-\frac{n-1}{m}\lp\nabla\beta,\nabla\phi\rp$.  Thus the gradient flow along $\nabla\beta$ yields a diffeomorphism $M=\bR\times N$ for $N$ a compact manifold.  The form of the metric and the measure then follows immediately from the conditions on $\nabla^2\beta$ and $\lp\nabla\beta,\nabla\phi\rp$.
\end{proof}

In fact, the above proof of Theorem~\ref{thm:splitting} given above also contains the essential estimate necessary to show that smooth metric measure spaces which satisfy the hypotheses of Theorem~\ref{thm:splitting} realize equality in Theorem~\ref{thm:weighted_cheng} (cf.\ \cite{Wang2001a}).

\begin{prop}
\label{prop:baby_lee}
Let $(M^n,g,v^m\dvol)$ be a smooth metric measure space with $\Ric_\phi^m\geq -\frac{1}{m+n-1}g$; if $m=\infty$, assume additionally that $\lv\nabla\phi\rv^2\leq1$.  Fix a base point $o\in M$ and suppose that there is a sequence of compact hypersurfaces $\{\Sigma_k\}$ such that
\begin{enumerate}
\item each $\Sigma_k$ separates $M$,
\item $d(o,\Sigma_k)\to\infty$ as $k\to\infty$, and
\item the quantity
\[ h_k = \min_{x\in\Sigma_k}\left\{ H_\phi(x), 1\right\}, \]
where $H_\phi$ is the weighted mean curvature of $\Sigma_k$, satisfies
\[ \lim_{k\to\infty} (1-h_k)e^{\frac{2d(o,\Sigma_k)}{m+n-1}} = 0 . \]
\end{enumerate}
Then $\lambda_1(-\Delta_\phi)=\frac{1}{4}$.
\end{prop}

\begin{proof}

From the proof of Theorem~\ref{thm:splitting}, we know that the Busemann function $\beta$ associated to $\{\Sigma_k\}$ satisfies $\Delta_\phi\beta\geq 1$ and $\lv\nabla\beta\rv^2\leq 1$ in the generalized support sense.  Let $\{\Omega_k\}$ be an exhaustion of $M$ by compact sets; that is, each $\Omega_k$ is a compact subset of $M$ with smooth boundary and $M=\bigcup_k\Omega_k$.  Let $\lambda_1(\Omega_k)$ be the first Dirichlet eigenvalue of $-\Delta_\phi$ on $\Omega_k$, and let $f$ be such that $-\Delta_\phi f = \lambda_1(\Omega_k)f$, $f\rv_{\partial\Omega_k}=0$, and $f>0$ in the interior of $\Omega_k$.  Define $h=fe^{\beta/2}$, and let $p\in\Omega_k$ be a point which realizes the maximum of $h$.  Without loss of generality, we may assume $\beta$ is smooth at $p$ (if not, simply argue using support functions).  It thus holds at $p$ that
\[ \nabla f = -\frac{1}{2}f\nabla\beta, \qquad 0 \geq \Delta_\phi\left(f e^{\beta/2}\right) . \]
Computing at $p$, we see that
\begin{align*}
0 & \geq \Delta_\phi f + \lp\nabla f,\nabla\beta\rp + \frac{1}{2}f\Delta_\phi\beta + \frac{1}{4}f\lv\nabla\beta\rv^2 \\
& = \left(-\lambda_1(\Omega_k) + \frac{1}{2}\Delta_\phi\beta - \frac{1}{4}\lv\nabla\beta\rv^2\right) f \\
& \geq \left(-\lambda_1(\Omega_k) + \frac{1}{4}\right)f .
\end{align*}
Thus $\lambda_1(\Omega_k)\geq\frac{1}{4}$ for all $k$.  The result follows by taking the limit $k\to\infty$.
\end{proof}
\section{On conformally compact smooth metric measure spaces}
\label{sec:ah}

Let us now describe in what way Theorem~\ref{thm:splitting} yields a connectedness result for asymptotically hyperbolic smooth metric measure spaces, analogous to the original motivation of Cai and Galloway~\cite{CaiGalloway1999} for studying Theorem~\ref{thm:cg}.  To that end, we first recall what it means for two smooth metric measure spaces to be pointwise conformally equivalent (cf.\ \cite{Case2010a}).

\begin{defn}
\label{defn:scms}
Two smooth metric measure spaces $(M^n,g,e^{-\phi}\dvol_g,m)$ and $(M^n,\hat g,e^{-\hat\phi}\dvol_{\hat g},m)$ are \emph{pointwise conformally equivalent} if there is a function $f\in C^\infty(M)$ such that
\begin{equation}
\label{eqn:scms}
\left( M^n, \hat g, e^{-\hat\phi}\dvol_{\hat g},m \right) = \left( M^n, e^{-\frac{2}{m+n-2}f}g, e^{-\frac{m+n}{m+n-2}f}e^{-\phi}\dvol_g, m \right) .
\end{equation}
\end{defn}

Unlike the Riemannian case $m=0$, there are a number of different ways one might define a ``weighted Yamabe constant'' (cf.\ \cite{Case2011gns}).  To avoid a discussion of this issue, we will instead formulate our weighted analogue of the connectedness result proven by Cai and Galloway~\cite{CaiGalloway1999} in terms of the weighted conformal Laplacian.

\begin{defn}
\label{defn:weighted_conformal_laplacian}
The \emph{weighted conformal Laplacian $L_\phi^m$} on a smooth metric measure space $(M^n,g,v^m\dvol)$ is the operator
\[ L_\phi^m := -\Delta_\phi + \frac{m+n-2}{4(m+n-1)}R_\phi^m . \]
\end{defn}

As an easy consequence of~\cite[Proposition~4.4]{Case2010a}, one has sees that the weighted conformal Laplacian is indeed a good generalization of the conformal Laplacian.

\begin{prop}
\label{prop:wcl_conformally_covariant}
Let $(M^n,g,v^m\dvol)$ be a smooth metric measure space.  The weighted conformal Laplacian $L_\phi^m$ is conformally covariant, in that for any $f\in C^\infty(M)$, the weighted conformal Laplacian $\widehat{L_\phi^m}$ of the smooth metric measure space $(M^n,\hat g,\hat v^m\dvol_{\hat g})$ defined by~\eqref{eqn:scms} satisfies
\[ \widehat{L_\phi^m} = e^{\frac{m+n+2}{2(m+n-2)}f} \circ L_\phi^m \circ e^{-\frac{1}{2}f} , \]
where the exponential factors above are to be regarded as multiplication operators.
\end{prop}

The two main facts about the weighted conformal Laplacian we will need are the following.

\begin{prop}
\label{prop:wcl_properties}
Let $(M^n,g,v^m\dvol)$ be a compact smooth metric measure space and denote by $\lambda_1(L_\phi^m)$ the first eigenvalue of the weighted conformal Laplacian; i.e.
\begin{equation}
\label{eqn:wcl_lambda1}
\lambda_1(L_\phi^m) = \inf\left\{ \frac{(L_\phi^mw,w)}{\lV w\rV_2^2} \colon 0\not=w\in W^{1,2}(M,v^m\dvol) \right\}
\end{equation}
for $(\cdot,\cdot)$ the $L^2(M,v^m\dvol)$-inner product and $\lV w\rV_2^2=(w,w)$.
\begin{enumerate}
\item The sign of the weighted conformal Laplacian is conformally covariant; i.e.\ $\lambda_1(L_\phi^m)$ is positive (resp.\ nonnegative) if and only if $\lambda_1(\widehat{L_\phi^m})$ is positive (resp.\ nonnegative) for $\widehat{L_\phi^m}$ as in Proposition~\ref{prop:wcl_conformally_covariant}.
\item There exists a positive function $w\in C^\infty(M)$ such that $L_\phi^mw=\lambda_1(L_\phi^m)w$.
\end{enumerate}
\end{prop}

\begin{proof}

By Proposition~\ref{prop:wcl_conformally_covariant}, we have that
\[ \left(\widehat{L_\phi^m}w,w\right)_{\hat v^m\dvol_{\hat g}} = \left( L_\phi^m( e^{-\frac{1}{2}f}w), e^{-\frac{1}{2}f}w \right)_{v^m\dvol_g} \]
for all $w\in C^\infty(M)$, where $(\cdot,\cdot)_{v^m\dvol}$ denotes the $L^2(M,v^m\dvol)$-inner product.  It follows immediately that the sign of $\lambda_1(L_\phi^m)$ is conformally invariant.

The existence of a positive function $w\in C^\infty(M)$ such that $L_\phi^mw=\lambda_1(L_\phi^m)w$ follows from the usual variational argument.  Namely, let $\{w_i\}\subset C^\infty(M)$ be a sequence of positive functions such that $\lV w_i\rV_2=1$ which minimizes~\eqref{eqn:wcl_lambda1}.  It follows that there exists a nonnegative $w\in W^{1,2}(M,v^m\dvol)$ such that $w_i\to w$ weakly in $W^{1,2}(M,v^m\dvol)$ and $w_i\to w$ strongly in $L^2(M,v^m\dvol)$.  It then follows from elliptic regularity and the maximum principle that $w\in C^\infty(M)$ is positive.
\end{proof}

Since all ``reasonable'' definitions of the weighted Yamabe constant have the property that the sign of the weighted Yamabe constant agrees with the sign of the first eigenvalue of the weighted conformal Laplacian, there is no problem using the latter in the formulation of Theorem~\ref{thm:connectedness}.

Before we define conformally compact quasi-Einstein smooth metric measure spaces, we first discuss two key examples of such spaces.

\begin{example}
\label{ex:qe2_zero}
Fix $3\leq n\in\bN$, $m\in[0,\infty)$, and $(F^{n-1},h)$ a compact Ricci flat manifold.  The smooth metric measure space
\begin{equation}
\label{eqn:zero_poincare_einstein}
\left( (-\infty,m+n-1] \times F^{n-1}, \rho^{-2}(t)\og, \rho^{-m-n}(t)\,\dvol_{\og}, m \right)
\end{equation}
with $t$ the coordinate on $(-\infty,m+n-1]$ and
\[ \og = dt^2\oplus h, \qquad \rho(t) = 1-\frac{t}{m+n-1} \]
is quasi-Einstein with quasi-Einstein constant $\lambda=-\frac{1}{m+n-1}$ and characteristic constant $\mu=0$.

As $m\to\infty$, this converges to the steady gradient Ricci soliton
\[ \left( \bR^n=\bR\times F^{n-1}, dt^2\oplus h, e^{t}\dvol \right) \]
\end{example}

This is nothing more than Example~\ref{ex:hyperbolic_cusp} written in different coordinates.  The key point of this presentation is that it is clearly conformally compact with defining function $\rho$ in the sense suggested by Definition~\ref{defn:scms}; that is, changing~\eqref{eqn:zero_poincare_einstein} conformally via the conformal factor $e^{-\frac{2}{m+n-2}f}=\rho^2$ in the sense of~\eqref{eqn:scms} yields a smooth metric measure space with metric and measure which extend smoothly to a nondegenerate metric and measure on the boundary $\{t=m+n-1\}$.

\begin{example}
\label{ex:qe2_negative}
Fix $3\leq n\in\bN$ and $m\in[0,\infty]$, set $k=\sqrt{m+n-1}$ and let $(S^{n-1},d\theta^2)$ be the standard $(n-1)$-sphere with its metric of constant sectional curvature one.  Then
\begin{equation}
\label{eqn:negative_poincare_einstein}
\left( \left[0,\frac{k\pi}{2}\right]\times S^{n-1}, \rho^{-2}(t)\og, \rho^{-m-n}(t)\,\dvol_{\og}, m \right)
\end{equation}
with $t$ the coordinate on $[0,\frac{k\pi}{2}]$ and
\[ \og = dt^2\oplus\left(k\sin\frac{t}{k}\right)^2d\theta^2, \qquad \rho(t) = \cos\frac{t}{k} \]
is quasi-Einstein with quasi-Einstein constant $\lambda=-1$ and characteristic constant $\mu=-\frac{m-1}{m+n-1}$.

As $m\to\infty$, this converges to the expanding Gaussian shrinker
\[ \left( \bR^n, dr^2\oplus r^2d\theta^2, e^{\frac{r^2}{2}}\dvol \right) . \]
\end{example}

In the terminology of~\cite{Case2010a,HePetersenWylie2010}, this is the negative elliptic Gaussian, and is again easily seen to be conformally compact with defining function $\rho$.

Both Example~\ref{ex:qe2_zero} and Example~\ref{ex:qe2_negative} can be regarded as natural weighted versions of the hyperbolic metric, in that (provided one chooses $(F^{n-1},h)$ to be Euclidean space in Example~\ref{ex:qe2_zero}) they are the ``weighted conformally flat'' quasi-Einstein smooth metric measure spaces with negative characteristic constant (cf.\ \cite{Case2011t,HePetersenWylie2011c}).  However, we see that they have quite different behavior in the limit $m=\infty$, necessitating a treatment of conformally compact quasi-Einstein smooth metric measure spaces which takes into account the characteristic constant as a parameter.  Further evidence of this comes from Theorem~\ref{thm:connectedness}, which requires \emph{both} the characteristic constant and the sign of the first eigenvalue $\lambda_1(L_\phi^m)$ of the conformal boundary to be nonnegative.

To begin the proof of Theorem~\ref{thm:connectedness}, let us first make precise what we mean by a conformally compact quasi-Einstein smooth metric measure space, as already anticipated in our discussion of Example~\ref{ex:qe2_zero} and Example~\ref{ex:qe2_negative}.

\begin{defn}
\label{defn:ccqe}
Suppose that $m\in[0,\infty)$.  We say that $(M^n,g,v^m\dvol)$ is a \emph{conformally compact quasi-Einstein smooth metric measure space} if:
\begin{enumerate}
\item $(M^n,g,v^m\dvol)$ is complete and quasi-Einstein.
\item There is a compact manifold $\overline{M}^n$ with boundary $\Sigma^{n-1}=\partial\overline{M}^n$ such that $M=\overline{M}\setminus\Sigma$.
\item There is a smooth metric $\og$ on $\overline{M}$, a positive function $\ov\in C^\infty(M)$, and a nonnegative function $\rho\in C^\infty(\overline{M})$ such that $\rho^{-1}(0)=\Sigma$, $d\rho\rv_\Sigma\not=0$, and
\[ \left( M^n, \og, \ov^m\dvol_{\og} \right) = \left( M^n, \rho^2g, \rho^{m+n}v^m\dvol_g \right) . \]
We call $\rho$ a \emph{defining function for $M$}.
\end{enumerate}

In this case, we call $[\Sigma,\og\rv_{T\Sigma},\ov^m\dvol]$, which is the equivalence class of all smooth metric measure spaces which are conformally equivalent to $(\Sigma,\og\rv_{T\Sigma},\ov^m\dvol)$, the \emph{conformal boundary} of $M$.
\end{defn}

It is clear from the definition that if $(M^n,g,v^m\dvol)$ is a conformally compact quasi-Einstein smooth metric measure space and $\rho$ is a defining function for $M$, then for any $\sigma\in C^\infty(\overline{M})$, the function $e^\sigma\rho$ is also a defining function for $M$.  In this way, we see that only $[\Sigma,\og\rv_{T\Sigma},\ov^m\dvol]$, and not $(\Sigma,\og\rv_{T\Sigma},\ov^m\dvol)$, is determined by $(M^n,g,v^m\dvol)$.

Our difficulty in extending Definition~\ref{defn:ccqe} to a meaningful definition in the limiting case $m=\infty$ is the present uncertainty as to what should be meant by the ``conformal boundary'' in this case.

As the following lemma shows, one can in fact regard conformally compact quasi-Einstein smooth metric measure spaces as ``asymptotically hyperbolic.''

\begin{lem}
\label{lem:weighted_ah}
Let $(M^n,g,v^m\dvol)$ be a conformally compact quasi-Einstein smooth metric measure space with quasi-Einstein constant $\lambda$, characteristic constant $\mu$, and defining function $\rho$.  Then it holds that
\[ \lv\onabla\rho\rv_{\og}^2 = -\frac{\lambda}{m+n-1} \]
on $\Sigma$.
\end{lem}

\begin{proof}

From~\cite[Proposition~4.4]{Case2010a}, we have that the weighted scalar curvatures $R_\phi^m$ and $\overline{R_\phi^m}$ of $(M^n,g,v^m\dvol)$ and its compactification $(\overline{M^n},\og,\ov^m\dvol)$ are related by
\[ R_\phi^m = \rho^2\overline{R_\phi^m} + 2(m+n-1)\rho\overline{\Delta_\phi}\rho - (m+n)(m+n-1)\lv\onabla\rho\rv_{\og}^2 . \]
Since $(M^n,g,v^m\dvol)$ is quasi-Einstein with quasi-Einstein constant $\lambda$ and characteristic constant $\mu$, it follows that
\begin{equation}
\label{eqn:conf_rphim}
\rho^2\left(\overline{R_\phi^m} + m\mu\ov^{-2}\right) + 2(m+n-1)\rho\overline{\Delta_\phi}\rho - (m+n)(m+n-1)\lv\onabla\rho\rv_{\og}^2 = (m+n)\lambda
\end{equation}
in $M$.  Since $\og$, $\ov$, and $\rho$ are all smooth in $\overline{M}$, we may evaluate~\eqref{eqn:conf_rphim} on $\Sigma$ to yield the result.
\end{proof}

Like asymptotically hyperbolic manifolds, given any choice of metric on the conformal boundary of a conformally compact quasi-Einstein manifold, there exists a choice of defining function for which the gradient has constant norm with respect to $\og$ in a neighborhood of $\Sigma$ (cf.\ \cite{Lee1995}).

\begin{prop}
\label{prop:geodesic_defining_function}
Let $(M^n,g,v^m\dvol)$ be a conformally compact quasi-Einstein smooth metric measure space with quasi-Einstein constant $\lambda$ and fix a representative $(\Sigma,\tilde h,\tilde v^m\dvol)$ of the conformal boundary of $(M^n,g,v^m\dvol)$.  Then there is a defining function $\rho\in C^\infty(\overline{M})$ such that the compactification
\begin{equation}
\label{eqn:conformal_compactification_formula}
\left( \overline{M}^n, \og, \ov^m\dvol_{\og} \right) = \left( \overline{M}^n, \rho^2g, \rho^{m+n}v^m\dvol_g \right)
\end{equation}
of $(M^n,g,v^m\dvol)$ satisfies $\og\rv_{T\partial\Sigma}=\tilde h$, $\ov\rv_\Sigma=\tilde v$, and
\begin{equation}
\label{eqn:geodesic_defining_fn}
\lv\onabla\rho\rv_{\og}^2 = -\frac{\lambda}{m+n-1}
\end{equation}
in a neighborhood of $\Sigma$.
\end{prop}

\begin{proof}

Let $\tilde\rho$ be a defining function such that the compactification
\[ \left( \overline{M}^n, \tilde{\bar g}, \tilde{\bar v}^m\dvol_{\tilde\bar g} \right) = \left( \overline{M}^n, \tilde\rho^2g, \tilde\rho^{m+n}v^m\dvol_g \right) \]
of $(M^n,g,v^m\dvol)$ satisfies $\tilde{\bar g}\rv_{T\partial\Sigma}=\tilde h$ and $\tilde{\bar v}\rv_\Sigma=\tilde v$.  Given any $\sigma\in C^\infty(\overline{M})$ such that $\sigma\rv_\Sigma=0$, the defining function $\rho=e^\sigma\tilde\rho$ will have the same property.  On the other hand, we compute that
\[ \lv d\rho\rv_{\rho^2g}^2 = e^{-2\sigma}\lv d(e^\sigma\tilde\rho)\rv_{\tilde{\bar g}}^2 = \lv d\tilde\rho\rv_{\tilde{\bar g}}^2 + 2\tilde\rho\lp d\sigma, d\tilde\rho\rp_{\tilde{\bar g}} + \tilde\rho^2\lv d\sigma\rv_{\tilde{\bar g}}^2 . \]
Thus~\eqref{eqn:geodesic_defining_fn} holds if and only if
\begin{equation}
\label{eqn:ode}
2\tilde\rho\lp d\sigma, d\tilde\rho\rp_{\tilde{\bar g}} + \tilde\rho^2\lv d\sigma\rv_{\tilde{\bar g}}^2 = -\frac{\lambda}{m+n-1} - \lv d\tilde\rho\rv_{\tilde{\bar g}}^2 .
\end{equation}
This is a first order noncharacteristic PDE in the unknown $\sigma$, and thus there exists a solution of~\eqref{eqn:ode} with initial condition $\sigma\rv_\Sigma=0$ in some neighborhood of $\Sigma$.
\end{proof}

Using this result, a straightforward modification of an argument by Cai and Galloway~\cite[Section~3]{CaiGalloway1999} allows us to show that the smooth metric measure spaces of Theorem~\ref{thm:connectedness} satisfy the hypotheses of Theorem~\ref{thm:splitting}.

\begin{prop}
\label{prop:hypersurfaces_to_infinity}
Let $(M^n,g,v^m\dvol)$ be a conformally compact quasi-Einstein smooth metric measure space with quasi-Einstein constant $\lambda=-\frac{1}{m+n-1}$ and fix $o\in M$.  Suppose that the sign of the first eigenvalue of the weighted conformal Laplacian of the conformal boundary of $(M^n,g,v^m\dvol)$ is nonnegative.  If $m>0$, suppose additionally that the characteristic constant of $(M^n,g,v^m\dvol)$ is nonnegative.  Then there exist a sequence $\{\Sigma_k\}\subset M$ of compact hypersurfaces such that
\begin{enumerate}
\item each $\Sigma_k$ separates $M$,
\item $d(o,\Sigma_k)\to\infty$ as $k\to\infty$, and
\item the quantity
\[ h_k = \min_{x\in\Sigma_k}\left\{ H_\phi(x), 1\right\}, \]
where $H_\phi$ is the weighted mean curvature of $\Sigma_k$, satisfies
\begin{equation}
\label{eqn:mean_curvature_decay}
\lim_{k\to\infty} (1-h_k)e^{\frac{2d(o,\Sigma_k)}{m+n-1}} = 0 .
\end{equation}
\end{enumerate}
Moreover, if either the weighted conformal Laplacian of the conformal boundary of $(M^n,g,v^m\dvol)$ is positive or if $m>0$ and the characteristic constant of $(M^n,g,v^m\dvol)$ is positive, then $H_\phi>1$ for all $k$ sufficiently large.
\end{prop}

\begin{proof}

Let $(\Sigma,\tilde g_1, \tilde v_1^m\dvol)$ be a choice of representative of the conformal boundary of $(M^n,g,v^m\dvol)$ and denote by $\widetilde{L_\phi^m}$ the weighted conformal Laplacian of $(\Sigma,\tilde g_1,\tilde v_1^m\dvol)$.  By Proposition~\ref{prop:wcl_properties}, there is a positive function $w\in C^\infty(\Sigma)$ such that $\widetilde{L_\phi^m}w=\lambda_1(\widetilde{L_\phi^m})w$.  Proposition~\ref{prop:wcl_conformally_covariant} implies that the weighted scalar curvature $\widetilde{R_\phi^m}$ of the representative
\begin{equation}
\label{eqn:w_scms}
\left( \Sigma,\tilde g,\tilde v^m\dvol_{\tilde g} \right) = \left( \Sigma,w^{\frac{4}{m+n-2}}\tilde g_1, w^{\frac{2(m+n)}{m+n-2}}\tilde v_1^m\dvol_{\tilde g_1} \right)
\end{equation}
of the conformal boundary of $(M^n,g,v^m\dvol)$ is such that
\[ \widetilde{R_\phi^m}=\frac{4(m+n-1)}{m+n-2}\lambda_1(\widetilde{L_\phi^m})w^{-\frac{4}{m+n-2}}, \]
which in particular has the same sign as the sign of the first eigenvalue of the weighted conformal Laplacian of the conformal boundary of $(M^n,g,v^m\dvol)$.

Now, let $\rho$ be the defining function associated to $(\Sigma,\tilde g,\tilde v^m\dvol)$ as per Proposition~\ref{prop:geodesic_defining_function}.  Thus the coordinate $r=(m+n-1)\rho$ is such that $\lv dr\rv_{\og}^2=1$ in a neighborhood of $\Sigma$, and so we have that
\[ g = \left(\frac{m+n-1}{r}\right)^2\left( dr^2\oplus h_r \right) \]
in this neighborhood, where $h_r$ is a smooth family of metrics on the level sets of $r$.  Denote these level sets by $\Sigma_r$, and denote by $\widehat{R_\phi^m}$, $\eta$, $H_\phi$, and $B$ the weighted scalar curvature, outward-pointing normal, weighted mean curvature, and second fundamental form of $(\Sigma_r,g\rv_{T\Sigma_r},(v\rv_{\Sigma_r})^m\dvol)$ in $(M^n,g,v^m\dvol)$.  As an immediate consequence of the Cauchy--Schwarz inequality and the weighted Gauss equation~\eqref{eqn:weighted_gauss}, we have that
\begin{equation}
\label{eqn:weighted_gauss_step1}
\frac{m+n-2}{m+n-1}(H_\phi)^2 \geq \widehat{R_\phi^m} + 2\Ric_\phi^m(\eta,\eta) - R_\phi^m .
\end{equation}
Using that $(M^n,g,v^m\dvol)$ is quasi-Einstein with quasi-Einstein constant $\lambda=-\frac{1}{m+n-1}$ and characteristic constant $\mu$, \eqref{eqn:weighted_gauss_step1} becomes
\begin{equation}
\label{eqn:weighted_gauss_step2}
\frac{m+n-2}{m+n-1}\left((H_\phi)^2 - 1 \right) \geq \rho^2\left( \overline{R_\phi^m} + m\mu\ov^{-2}\right)
\end{equation}
for $\overline{R_\phi^m}=\rho^{-2}\widehat{R_\phi^m}$ the weighted scalar curvature of $(\Sigma_r,\og_{T\Sigma_r},(\ov_{\Sigma_r})^m\dvol)$.  Dividing both sides of~\eqref{eqn:weighted_gauss_step2} by $\rho^2$ and taking the limit $r\to0$ yields
\begin{equation}
\label{eqn:weighted_gauss_step3}
\lim_{r\to 0} \left((H_\phi)^2 - 1\right)\rho^{-2} \geq 0,
\end{equation}
where we have now used the assumptions $\mu\geq 0$ and $\lambda_1(\widetilde{L_\phi^m})\geq 0$.  Since the weighted mean curvature of $\Sigma_r$ is positive for $r$ sufficiently small, it follows from~\eqref{eqn:weighted_gauss_step3} that
\begin{equation}
\label{eqn:weighted_gauss_step4}
\lim_{r\to 0} \left(H_\phi - 1\right)\rho^{-2} \geq 0 .
\end{equation}
Finally, fixing $r_0>0$ sufficiently small and $o\in\Sigma_{r_0}$, we compute for $0<r<r_0$ that
\[ \exp\left(d(o,\Sigma_r)\right) = \exp\left(\int_r^{r_0} \frac{m+n-1}{s}ds\right) = \left(\frac{r_0}{r}\right)^{m+n-1} . \]
Using this to write $\rho$ in~\eqref{eqn:weighted_gauss_step4} in terms of $d(o,\Sigma_r)$ then yields~\eqref{eqn:mean_curvature_decay}.

Finally, if either $\lambda_1(\widetilde{L_\phi^m})>0$ or $m\mu>0$, the inequality in~\eqref{eqn:weighted_gauss_step3} is strict, yielding the last claim of the proposition.
\end{proof}

Theorem~\ref{thm:connectedness} follows almost immediately from Proposition~\ref{prop:hypersurfaces_to_infinity}.

\begin{proof}[Proof of Theorem~\ref{thm:connectedness}]

Suppose that $M$ has at least two ends $E_1$ and $E_2$.  By Proposition~\ref{prop:hypersurfaces_to_infinity}, there is a sequence of hypersurfaces $\{\Sigma_k\}\subset M$ satisfying the hypotheses of Theorem~\ref{thm:splitting}, and hence there is a compact Ricci flat manifold $(N^{n-1},h)$ such that $(M^n,g,v^m\dvol)$ is isometric to
\[ \left( \bR\times N^{n-1}, dt^2\oplus e^{\frac{2t}{m+n-1}}h, e^{\frac{mt}{m+n-1}}\dvol \right) . \]
However, the end $(-\infty,0]\times N$ is not conformally compact, a contradiction.
\end{proof}

As another corollary of Proposition~\ref{prop:hypersurfaces_to_infinity}, we have the following Lee-type~\cite{Lee1995} result on the bottom of the spectrum of the weighted Laplacian of a certain class of conformally compact quasi-Einstein smooth metric measure spaces (see also~\cite{Wang2001a}).  In particular, Theorem~\ref{thm:connectedness} is also a consequence of Theorem~\ref{thm:weighted_cheng}.

\begin{thm}
\label{thm:ccqe_to_spectrum}
Let $(M^n,g,v^m\dvol)$ be a conformally compact quasi-Einstein smooth metric measure space with quasi-Einstein constant $\lambda=-\frac{1}{m+n-1}$ such that the first eigenvalue of the weighted conformal Laplacian of the conformal boundary of $(M^n,g,v^m\dvol)$ is nonnegative.  If $m>0$, suppose additionally that the characteristic constant of $(M^n,g,v^m\dvol)$ is nonnegative.  Then
\[ \lambda_1(-\Delta_\phi) = \frac{1}{4} . \]
\end{thm}

\begin{proof}

This follows immediately from Proposition~\ref{prop:hypersurfaces_to_infinity} and Proposition~\ref{prop:baby_lee}.
\end{proof}

\bibliographystyle{abbrv}
\bibliography{../bib}
\end{document}